\documentclass[11pt]{article}
\usepackage[a4paper,margin=1in]{geometry}
\usepackage[T1]{fontenc}
\usepackage{lmodern}
\usepackage{amsmath,amssymb,mathtools,amsthm}
\usepackage{graphicx}
\usepackage{listings}
\usepackage{microtype}
\usepackage{cite}
\newtheorem{lemma}{Lemma}
\newtheorem{theorem}{Theorem}

\DeclareMathOperator{\sgn}{sgn}
\newcommand{\ii}{\mathrm i}
\newcommand{\dd}{\,\mathrm d}
\newcommand{\Rea}{\operatorname{Re}}
\newcommand{\Ima}{\operatorname{Im}}
\newcommand{\br}[3]{\left[#1\right]_{#2}^{#3}}
\title{\vspace{-12mm}\textbf{Explicit Cartesian Coordinates for Three-Dimensional Clothoids}}
\author{Alexandru Ionu\textcommabelow{t}}
\date{}

\begin{document}
\maketitle
\vspace{-6mm}
\begin{abstract}
We combine the spherical-clothoid reconstruction theorem of Lucas and Ortega-Yag\"ues with an explicit spinorial parametrization of the spherical clothoid and a Kamp\'e de F\'eriet evaluation of Kummer-product integrals. This gives explicit Cartesian coordinates, in arclength, for every nonhelical three-dimensional clothoid. Kummer asymptotics then yield explicit asymptotic lines, including their gamma--digamma offsets and the angle between their axes.
\end{abstract}

\section{Introduction}
A three-dimensional clothoid is a unit-speed space curve whose curvature and torsion are affine functions of arclength,
\begin{equation}
\kappa(s)=cs+d,\qquad \tau(s)=as+b. \tag{1}
\end{equation}
Throughout, \(\kappa(s)=cs+d\) denotes the signed coefficient in the smoothly continued Frenet equations; away from its zeros, the ordinary curvature is \(|cs+d|\).
Berry and Robbins earlier obtained an exact spinorial parametrization for the special case of constant curvature and linearly varying torsion, through its equivalence with the Landau--Zener problem \cite{BerryRobbins1993}.
Frego obtained closed Fresnel formulas for an analytically tractable subfamily and treated the remaining cases numerically \cite{Frego2022}. More recently, Lucas and Ortega-Yag\"ues proved that every nonhelical curve satisfying (1) can be reconstructed from a spherical clothoid, but their formula retains one vector quadrature of the spherical normal \cite[Theorem~4.4]{LucasOrtega2026}.

The Sabban equations admit a two-component spinor formulation \cite{Senyurt2015}, and the spherical clothoid together with its full Sabban frame has an explicit representation in confluent hypergeometric functions \cite{Ionut2022}. Substitution of that frame into the reconstruction theorem turns the formula itself into a single integral of spinor quadratics. The remaining Kummer-product integrals are then evaluated by the Kamp\'e de F\'eriet series studied by Jur\v{s}\.{e}nas \cite{Jursenas2014}.

\section{The spinorial integral form of Theorem 4.4}
Following \cite{LucasOrtega2026}, put
\begin{equation}
\Delta=ad-bc,\qquad \varepsilon=\sgn\Delta,\qquad \rho=\sqrt{a^2+c^2}, \tag{2}
\end{equation}
and assume \(\Delta\ne0\). The constants in their Theorem~4.4 are
\begin{equation}
 p=\frac{\varepsilon a\Delta}{\rho^3},\quad
 q=-\frac{\varepsilon c}{\rho},\quad
 \lambda=\frac{\varepsilon\Delta}{\rho},\quad
 m=\frac{\varepsilon\rho^3}{\Delta^2},\quad
 n=\frac{ab+cd}{\Delta}. \tag{3}
\end{equation}
If \(\delta\) is the unit-speed spherical clothoid with geodesic curvature \(\kappa_g(u)=mu+n\) and intrinsic spherical normal \(N_\delta\), then, up to an orientation-preserving Euclidean congruence,
\begin{equation}
 r(s)=p\bigl(\kappa_g(\lambda s)\delta(\lambda s)+N_\delta(\lambda s)\bigr)
      +q\int_{s_0}^{s}N_\delta(\lambda u)\dd u. \tag{4}
\end{equation}
Introduce
\begin{equation}
 \chi=-\frac{\Delta}{\rho^{3/2}},\qquad
 \xi(s)=\sqrt\rho\left(s+\frac{ab+cd}{\rho^2}\right),\qquad \xi_0=\xi(0), \tag{5}
\end{equation}
so that
\begin{equation}
 \kappa_g(\lambda s)=m\lambda s+n=-\frac{\xi(s)}{\chi}. \tag{6}
\end{equation}
Set
\begin{equation}
 \mu=\frac12+\frac{\ii\chi^2}{8},\qquad \zeta=\frac{\ii\xi^2}{2},\qquad
 U(\xi)={}_1F_1\!\left(\mu;\frac12;\zeta\right),\qquad
 V(\xi)={}_1F_1\!\left(\mu;\frac32;\zeta\right), \tag{7}
\end{equation}
\begin{equation}
 \psi_1(\xi)=e^{-\zeta/2}U(\xi),\qquad
 \psi_2(\xi)=\frac{\ii\chi\xi}{2}e^{-\zeta/2}V(\xi). \tag{8}
\end{equation}
For \(v=(v_1,v_2,v_3)\in\mathbb R^3\), write \(v^\perp:=v_1+\ii v_2\), while \(v_3\) denotes its third component. We use the Sabban ordering
\(\{\gamma,\mathbf t,\mathbf d\}\) and convention \(\mathbf d=\gamma\times\mathbf t\) of
\cite{Senyurt2015}. Equivalently, the spinor representing
\(\{\mathbf t,\mathbf d,\gamma\}\) in the convention of that paper is
\(e^{-\pi\ii/4}(\overline{\psi_1},-\psi_2)^T\). The complete frame is
\begin{align}
 \gamma^\perp&=-2\psi_1\psi_2,& \gamma_3&=|\psi_1|^2-|\psi_2|^2, \tag{9}\\
 \mathbf t^\perp&=\ii(\psi_1^2+\psi_2^2),& (\mathbf t)_3&=2\Ima(\overline{\psi_1}\psi_2), \notag\\
 \mathbf d^\perp&=-(\psi_1^2-\psi_2^2),& (\mathbf d)_3&=-2\Rea(\overline{\psi_1}\psi_2). \tag{10}
\end{align}
Choose the orientation so that the notation of \cite{LucasOrtega2026} is
\begin{equation}
 \delta(\lambda s)=\varepsilon\gamma(\xi(s)),\qquad
 T_\delta(\lambda s)=\mathbf t(\xi(s)),\qquad
 N_\delta(\lambda s)=\varepsilon\mathbf d(\xi(s)). \tag{11}
\end{equation}
In the variable \(\xi\), the pulled-back Sabban equations are
\begin{equation}
 \gamma'=-\chi\mathbf t,\qquad
 \mathbf t'=\chi\gamma+\xi\mathbf d,\qquad
 \mathbf d'=-\xi\mathbf t,
 \qquad\text{hence}\qquad (\xi\gamma-\chi\mathbf d)'=\gamma. \tag{12}
\end{equation}
Substituting (3)--(11) into (4), translating so that \(r(0)=0\), and using (12), gives the direct integral form
\begin{equation}
 \boxed{\displaystyle
 r(s)=\rho^{-3/2}\int_{\xi_0}^{\xi(s)}\bigl(a\,\gamma(\eta)-c\,\mathbf d(\eta)\bigr)\dd\eta.} \tag{13}
\end{equation}
Thus Theorem~4.4 is already an explicit spinorial quadrature. In Cartesian form,
\begin{align}
 x(s)+\ii y(s)
 &=\rho^{-3/2}\int_{\xi_0}^{\xi(s)}
 \Bigl[c(\psi_1^2-\psi_2^2)-2a\psi_1\psi_2\Bigr]\dd\eta, \tag{14}\\
 z(s)
 &=\rho^{-3/2}\int_{\xi_0}^{\xi(s)}
 \Bigl[2c\Rea(\overline{\psi_1}\psi_2)+a(|\psi_1|^2-|\psi_2|^2)\Bigr]\dd\eta, \tag{15}
\end{align}
where the spinors in the integrands are evaluated at \(\eta\). Equation (12) also shows that only the normal-field part requires a new primitive:
\begin{equation}
 r(s)=\rho^{-3/2}\br{cJ_\chi(\xi)+a\bigl(\xi\gamma(\xi)-\chi\mathbf d(\xi)\bigr)}{\xi_0}{\xi(s)},
 \qquad J_\chi'=-\mathbf d. \tag{16}
\end{equation}

\section{Evaluation of the normal-field primitive}
We use the convention
\begin{equation}
F_{1:1;1}^{1:1;1}\!\left[
\begin{matrix}A:B;C\\D:E;F\end{matrix};X,Y\right]
=\sum_{j,k\ge0}\frac{(A)_{j+k}(B)_j(C)_k}{(D)_{j+k}(E)_j(F)_k}\frac{X^jY^k}{j!k!}. \tag{17}
\end{equation}
Define
\begin{align}
\mathcal F_\chi(\xi)
&:=F_{1:1;1}^{1:1;1}\!\left[
\begin{matrix}\frac12:\mu;\frac12-\mu\\[1mm]\frac32:\frac12;\frac12\end{matrix};\zeta,-\zeta\right], \tag{18}\\
\mathcal G_\chi(\xi)
&:=F_{1:1;1}^{1:1;1}\!\left[
\begin{matrix}1:\mu;1-\mu\\[1mm]2:\frac32;\frac12\end{matrix};\zeta,-\zeta\right]. \tag{19}
\end{align}
\begin{lemma}
For real \(\xi\),
\begin{equation}
 (\xi\mathcal F_\chi)'=\psi_1^2,\qquad
 \left(\frac{\ii\chi\xi^2}{4}\mathcal G_\chi\right)'=\overline{\psi_1}\psi_2. \tag{20}
\end{equation}
Consequently,
\begin{equation}
 J_\chi^\perp=2\xi\mathcal F_\chi-\xi e^{-\zeta}UV,
 \qquad
 (J_\chi)_3=\Rea\left(\frac{\ii\chi\xi^2}{2}\mathcal G_\chi\right). \tag{21}
\end{equation}
\end{lemma}
\begin{proof}
The identities in (20) are the relevant specializations of the Kummer-product integral in \cite{Jursenas2014}; they also follow by multiplying the two \({}_1F_1\) series and integrating termwise. For \(N=j+k\),
\[
 \frac{(1/2)_N}{(3/2)_N}=\frac1{2N+1},\qquad
 \frac{(1)_N}{(2)_N}=\frac1{N+1},
\]
which cancel the derivatives of \(\xi^{2N+1}\) and \(\xi^{2N+2}/2\). Kummer's transformation, together with conjugation for real \(\xi\), gives the opposite-argument factors in (20). Finally,
\[
 (\psi_1\psi_2)'=\frac{\ii\chi}{2}(\psi_1^2+\psi_2^2)
\]
eliminates the third apparent primitive and yields (21).
\end{proof}

\section{Explicit Cartesian coordinates}
\begin{theorem}
Let (1) hold with \(a^2+c^2>0\) and \(\Delta\ne0\). With the notation (2), (5), (7), (18), and (19), the corresponding three-dimensional clothoid is, up to an orientation-preserving rigid motion,
\begin{align}
 x(s)+\ii y(s)
 &=\rho^{-3/2}\br{
 2c\xi\mathcal F_\chi+e^{-\zeta}\left\{
 a\chi U^2+\frac{a\chi^3\xi^2}{4}V^2-(c\xi+\ii a\chi\xi^2)UV
 \right\}}{\xi_0}{\xi(s)}, \tag{22}\\[1mm]
 z(s)
 &=\rho^{-3/2}\br{
 \Rea\left(\frac{\ii c\chi\xi^2}{2}\mathcal G_\chi\right)
 +a\xi\left\{2|U|^2-1-\chi^2\Ima(\overline{U}V)\right\}}{\xi_0}{\xi(s)}. \tag{23}
\end{align}
Every function in the brackets is evaluated at \(\xi\), and \(\br{f}{\xi_0}{\xi(s)}=f(\xi(s))-f(\xi_0)\).
\end{theorem}
\begin{proof}
Insert (9)--(10) and (21) into (16). For the third coordinate use \(|\psi_1|^2+|\psi_2|^2=1\) and
\(2\Rea(\overline{\psi_1}\psi_2)=-\chi\xi\Ima(\overline{U}V)\).
\end{proof}
The endpoint subtraction gives \(r(0)=0\); arbitrary initial position and orientation are obtained by a single rigid motion.

\section{Explicit asymptotic lines}
The integral representation (14)--(16), rather than the expanded Kamp\'e de F\'eriet expressions (22)--(23), is the convenient starting point for asymptotics. Define only the phase
\begin{equation}
 \Phi(\xi)=\frac{\xi^2}{4}+\frac{\chi^2}{4}\log|\xi|, \tag{24}
\end{equation}
and the two Kummer-branch coefficients
\begin{equation}
 A_\chi=
 \frac{\sqrt\pi\,e^{-\pi\chi^2/16}2^{-\ii\chi^2/8}}
 {\Gamma(\frac12+\frac{\ii\chi^2}{8})},\qquad
 B_\chi=
 \frac{\chi\sqrt\pi}{2\sqrt2}\,
 \frac{e^{-\pi\chi^2/16}2^{\ii\chi^2/8}e^{3\pi\ii/4}}
 {\Gamma(1-\frac{\ii\chi^2}{8})}. \tag{25}
\end{equation}

\begin{lemma}[Spinor asymptotics]
As \(\xi\to\pm\infty\), the spinors possess the completely explicit Poincar\'e expansions
\begin{align}
 \psi_1(\xi)\sim{}&A_\chi e^{\ii\Phi(\xi)}
 \sum_{n=0}^{\infty}
 \frac{(-2\ii)^n
 (\frac12-\frac{\ii\chi^2}{8})_n
 (-\frac{\ii\chi^2}{8})_n}
 {n!\,|\xi|^{2n}} \notag\\
 &-\frac{\chi B_\chi}{2|\xi|}e^{-\ii\Phi(\xi)}
 \sum_{n=0}^{\infty}
 \frac{(2\ii)^n
 (\frac12+\frac{\ii\chi^2}{8})_n
 (1+\frac{\ii\chi^2}{8})_n}
 {n!\,|\xi|^{2n}}, \tag{26}\\
 \psi_2(\xi)\sim{}&\ \pm B_\chi e^{-\ii\Phi(\xi)}
 \sum_{n=0}^{\infty}
 \frac{(2\ii)^n
 (\frac12+\frac{\ii\chi^2}{8})_n
 (\frac{\ii\chi^2}{8})_n}
 {n!\,|\xi|^{2n}} \notag\\
 &\pm\frac{\chi A_\chi}{2|\xi|}e^{\ii\Phi(\xi)}
 \sum_{n=0}^{\infty}
 \frac{(-2\ii)^n
 (\frac12-\frac{\ii\chi^2}{8})_n
 (1-\frac{\ii\chi^2}{8})_n}
 {n!\,|\xi|^{2n}}. \tag{27}
\end{align}
In particular,
\begin{align}
 \psi_1(\xi)&=A_\chi e^{\ii\Phi(\xi)}
 -\frac{\chi B_\chi}{2|\xi|}e^{-\ii\Phi(\xi)}
 +O(|\xi|^{-2}), \notag\\
 \psi_2(\xi)&=\pm B_\chi e^{-\ii\Phi(\xi)}
 \pm\frac{\chi A_\chi}{2|\xi|}e^{\ii\Phi(\xi)}
 +O(|\xi|^{-2}). \tag{28}
\end{align}
\end{lemma}

\begin{proof}
Apply the standard two-branch expansion of \({}_1F_1\) \cite[Eq.~13.7.2]{DLMF} to (7)--(8), with the principal logarithm of \(\zeta=\ii\xi^2/2\). The exponential and algebraic branches give (25), and the inverse-power factors give the two displayed series.
\end{proof}

The gamma-modulus identities imply
\begin{equation}
 |A_\chi|^2=\frac{1+e^{-\pi\chi^2/4}}2,
 \qquad
 |B_\chi|^2=\frac{1-e^{-\pi\chi^2/4}}2. \tag{29}
\end{equation}
Define the two limiting directions directly by
\begin{equation}
 (u_\pm)^\perp=
 \pm\frac{\chi\pi e^{-\pi\chi^2/8}e^{-\pi\ii/4}}
 {\sqrt2\,\Gamma(\frac12+\frac{\ii\chi^2}{8})
 \Gamma(1-\frac{\ii\chi^2}{8})},
 \qquad
 (u_\pm)_3=e^{-\pi\chi^2/4}. \tag{30}
\end{equation}
The same gamma identities show that \(|u_\pm|=1\). Substitution of (28) into (9)--(10) gives, as \(\xi\to\pm\infty\),
\begin{align}
 \gamma^\perp(\xi)={}&(u_\pm)^\perp
 -\frac{\chi}{\xi}\left(
 A_\chi^2e^{2\ii\Phi(\xi)}
 -B_\chi^2e^{-2\ii\Phi(\xi)}\right)
 +O(|\xi|^{-2}), \notag\\
 \gamma_3(\xi)={}&(u_\pm)_3
 -\frac{2\chi}{|\xi|}\Rea\!\left(
 \overline{A_\chi}B_\chi e^{-2\ii\Phi(\xi)}\right)
 +O(|\xi|^{-2}), \tag{31}\\
 \mathbf d^\perp(\xi)={}&-A_\chi^2e^{2\ii\Phi(\xi)}
 +B_\chi^2e^{-2\ii\Phi(\xi)}
 -\frac{\chi}{\xi}(u_\pm)^\perp+O(|\xi|^{-2}), \notag\\
 (\mathbf d)_3(\xi)={}&\mp2\Rea\!\left(
 \overline{A_\chi}B_\chi e^{-2\ii\Phi(\xi)}\right)
 -\frac{\chi}{\xi}(u_\pm)_3+O(|\xi|^{-2}). \tag{32}
\end{align}
In particular,
\begin{equation}
 \gamma(\xi)\longrightarrow u_\pm,
 \qquad
 \xi\gamma(\xi)-\chi\mathbf d(\xi)=\xi u_\pm+O(|\xi|^{-1}). \tag{33}
\end{equation}

It remains to determine the finite part of \(J_\chi\). Let \(\partial_\chi\) denote differentiation at fixed \(\xi\), and let \(\Psi=\Gamma'/\Gamma\). The Kummer pair (8) satisfies
\begin{equation}
 \psi_1'=\frac{\ii\xi}{2}\psi_1+\frac{\ii\chi}{2}\psi_2,
 \qquad
 \psi_2'=\frac{\ii\chi}{2}\psi_1-\frac{\ii\xi}{2}\psi_2. \tag{34}
\end{equation}
Differentiating (34) with respect to \(\chi\) at fixed \(\xi\), the diagonal terms cancel and give the exact primitive identities
\begin{align}
 J_\chi^\perp&=-2\ii\left(
 \psi_1\partial_\chi\psi_2-
 \psi_2\partial_\chi\psi_1\right), \tag{35}\\
 (J_\chi)_3&=-2\ii\left(
 \overline{\psi_1}\partial_\chi\psi_1+
 \overline{\psi_2}\partial_\chi\psi_2\right), \tag{36}
\end{align}
with \(J_\chi(0)=0\). The logarithmic derivatives of the explicit coefficients (25) are
\begin{align}
 \frac{A_\chi'}{A_\chi}
 &=\frac{\chi}{4}\left[-\frac\pi2-\ii\log2
 -\ii\Psi\!\left(\frac12+\frac{\ii\chi^2}{8}\right)\right], \tag{37}\\
 \frac{B_\chi'}{B_\chi}
 &=\frac1\chi+\frac{\chi}{4}\left[-\frac\pi2+\ii\log2
 +\ii\Psi\!\left(1-\frac{\ii\chi^2}{8}\right)\right]. \tag{38}
\end{align}
Since \(J_\chi'=-\mathbf d\), (32) shows that the nonoscillatory term \(\chi u_\pm/\xi\) produces the logarithm, while \(2\Phi'(\xi)=\xi+\chi^2/(2\xi)\) makes the oscillatory tails \(O(|\xi|^{-1})\) after one integration by parts. Substitution of (28) into (35)--(38) then yields
\begin{equation}
 J_\chi(\xi)=\chi u_\pm\log|\xi|+C_\pm+O(|\xi|^{-1}),
 \qquad \xi\to\pm\infty, \tag{39}
\end{equation}
where the finite parts are entirely explicit:
\begin{align}
 (C_\pm)^\perp={}&\pm
 \frac{\chi\pi e^{-\pi\chi^2/8}e^{-\pi\ii/4}}
 {\sqrt2\,\Gamma(\frac12+\frac{\ii\chi^2}{8})
 \Gamma(1-\frac{\ii\chi^2}{8})} \notag\\
 &\times\left[
 \frac{\ii}{\chi}-\frac{\chi}{4}\left\{
 2\log2+\Psi\!\left(1-\frac{\ii\chi^2}{8}\right)
 +\Psi\!\left(\frac12+\frac{\ii\chi^2}{8}\right)
 \right\}\right], \tag{40}\\
 (C_\pm)_3={}&\frac{\chi}{4}\Bigg[
 \left(1-e^{-\pi\chi^2/4}\right)
 \left\{\log2+\Rea\Psi\!\left(1-\frac{\ii\chi^2}{8}\right)\right\} \notag\\
 &\hspace{18mm}-\left(1+e^{-\pi\chi^2/4}\right)
 \left\{\log2+\Rea\Psi\!\left(\frac12+\frac{\ii\chi^2}{8}\right)\right\}
 \Bigg]. \tag{41}
\end{align}
Thus \((C_-)^\perp=-(C_+)^\perp\) and \((C_-)_3=(C_+)_3\).

\begin{theorem}[Explicit asymptotic lines]
Let the hypotheses of Theorem~1 hold, and set
\begin{equation}
 b_\pm=\rho^{-3/2}\left\{
 cC_\pm-cJ_\chi(\xi_0)
 -a\bigl(\xi_0\gamma(\xi_0)-\chi\mathbf d(\xi_0)\bigr)
 \right\}. \tag{42}
\end{equation}
Then, as \(s\to\pm\infty\),
\begin{align}
 r(s)={}&b_\pm+u_\pm\Bigg[
 \frac{a}{\rho}\left(s+\frac{ab+cd}{\rho^2}\right) \notag\\
 &\hspace{16mm}-\frac{c\Delta}{\rho^3}
 \log\left|\sqrt\rho\left(s+\frac{ab+cd}{\rho^2}\right)\right|
 \Bigg]+O(|s|^{-1}). \tag{43}
\end{align}
Consequently,
\begin{equation}
 \operatorname{dist}\bigl(r(s),\,b_\pm+\mathbb R u_\pm\bigr)
 =O(|s|^{-1}). \tag{44}
\end{equation}
The oriented angle \(\Theta\in[0,\pi]\) between the two asymptotic directions is
\begin{equation}
 \cos\Theta=2e^{-\pi\chi^2/2}-1,
 \qquad
 \Theta=2\arccos\!\left(e^{-\pi\chi^2/4}\right). \tag{45}
\end{equation}
\end{theorem}

\begin{proof}
Insert (33) and (39) into (16), use
\(\xi(s)=\sqrt\rho\bigl(s+(ab+cd)/\rho^2\bigr)\), and recall
\(\chi=-\Delta/\rho^{3/2}\). This gives (43) and the explicit basepoints (42). The logarithmic displacement is parallel to \(u_\pm\), so it does not affect the distance to the corresponding line, proving (44). Finally, (30) and the gamma-modulus identities give
\(u_+\cdot u_-=2e^{-\pi\chi^2/2}-1\), hence (45).
\end{proof}

\paragraph{Geometric interpretation.}
If \(c=0\), the logarithmic term vanishes and (43) is an ordinary affine asymptotic parametrization. If \(a=0\), the longitudinal escape along each asymptotic line is logarithmic. In every nonhelical case \(\Delta\ne0\), the transverse distance to the appropriate line decays as \(O(|s|^{-1})\). The acute angle between the unoriented lines is
\(\arccos|2e^{-\pi\chi^2/2}-1|\).

\paragraph{Numerical illustration.}
Figure~\ref{fig:explicit-example} revisits the example \(\kappa(s)=s+1\), \(\tau(s)=s/10\) used in \cite[Figure~1]{LucasOrtega2026}. Its Cartesian coordinates are evaluated directly from (22)--(23), rather than by numerical integration of the Frenet--Serret equations. The routine \texttt{mpmath.hyper2d} is especially convenient because its parameter dictionaries transcribe (18)--(19) directly; the asymptotic lines are evaluated independently from (30) and (40)--(42).

\begin{figure}[htbp]
\centering
\includegraphics[width=0.78\textwidth]{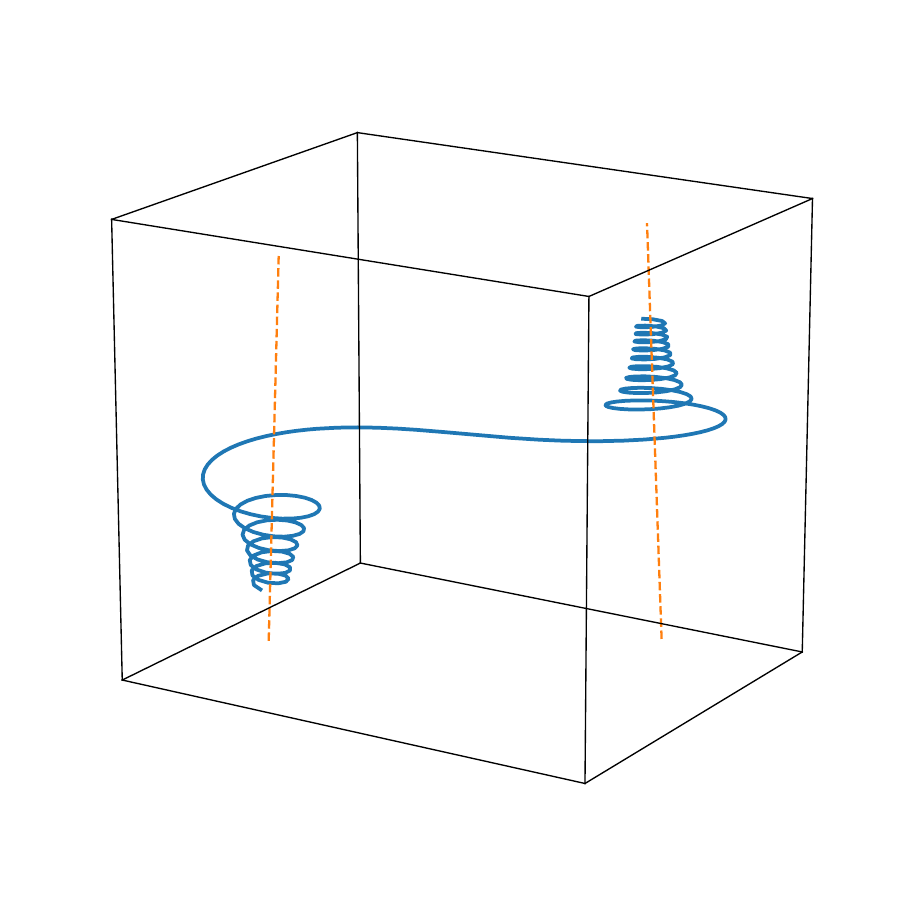}
\caption{The three-dimensional clothoid \(\kappa(s)=s+1\), \(\tau(s)=s/10\), \(-10\le s\le10\), together with its explicit asymptotic lines \(b_-+\mathbb R u_-\) and \(b_++\mathbb R u_+\), shown dashed.}
\label{fig:explicit-example}
\end{figure}

The excluded cases are already explicit: if \(a=c=0\), one obtains the elementary constant-curvature, constant-torsion helix; if \(\Delta=0\) and \((a,c)\ne(0,0)\), then \(\tau/\kappa\) is constant and the commuting clothoid-helix family reduces to Fresnel integrals \cite{Frego2022,Rosu2026}. Thus Theorems~1 and 2 complete both the closed-form Cartesian and asymptotic treatment of affine curvature and torsion.

\clearpage
\appendix
\section{Reproducible evaluation}
The following is the exact Python program used to generate Figure~\ref{fig:explicit-example}.
\begin{lstlisting}[language=Python]
#!/usr/bin/env python3
"""Reproduce Figure 1 directly from the formulas of the paper.

The curve is evaluated from equations (22)--(23), with the Kampe de Feriet
functions (18)--(19) evaluated by mpmath.hyper2d.  The asymptotic lines are
evaluated independently from (30) and (40)--(42).  No Frenet--Serret ODE is
integrated numerically.

Requires Python 3 with mpmath, numpy, and matplotlib.
"""

from pathlib import Path
import multiprocessing

import matplotlib.pyplot as plt
import mpmath as mp
import numpy as np


# Working precision is set before any mpmath constants are constructed.
mp.mp.dps = 40
ii = mp.j

# Numerical illustration: kappa(s)=c s+d=s+1, tau(s)=a s+b=s/10.
a, b, c, d = map(mp.mpf, ("0.1", "0", "1", "1"))
s_values = np.linspace(-10.0, 10.0, 301)

# Equations (2), (5), and (7).
Delta = a*d - b*c
rho = mp.sqrt(a*a + c*c)
chi = -Delta / rho**mp.mpf("1.5")
xi0 = mp.sqrt(rho) * (a*b + c*d) / rho**2
mu = mp.mpf("0.5") + ii*chi**2/8


def xi(s):
    """Equation (5)."""
    return mp.sqrt(rho) * (s + (a*b + c*d)/rho**2)


def zeta(xi_value):
    """zeta=i xi^2/2 in equation (7)."""
    return ii*xi_value**2/2


def U(xi_value):
    """U(xi) in equation (7)."""
    return mp.hyp1f1(mu, mp.mpf("0.5"), zeta(xi_value))


def V(xi_value):
    """V(xi) in equation (7)."""
    return mp.hyp1f1(mu, mp.mpf("1.5"), zeta(xi_value))


def psi1(xi_value):
    """psi_1(xi) in equation (8)."""
    return mp.exp(-zeta(xi_value)/2) * U(xi_value)


def psi2(xi_value):
    """psi_2(xi) in equation (8)."""
    return ii*chi*xi_value*mp.exp(-zeta(xi_value)/2) * V(xi_value)/2


def sabban_frame(xi_value):
    """The complete Sabban frame {gamma,t,d} in equations (9)--(10)."""
    p1 = psi1(xi_value)
    p2 = psi2(xi_value)

    gamma_perp = -2*p1*p2
    t_perp = ii*(p1**2 + p2**2)
    d_perp = -(p1**2 - p2**2)

    # Dictionary keys are the paper's frame symbols gamma, t, d.
    return {
        "gamma": mp.matrix([
            mp.re(gamma_perp),
            mp.im(gamma_perp),
            abs(p1)**2 - abs(p2)**2,
        ]),
        "t": mp.matrix([
            mp.re(t_perp),
            mp.im(t_perp),
            2*mp.im(mp.conj(p1)*p2),
        ]),
        "d": mp.matrix([
            mp.re(d_perp),
            mp.im(d_perp),
            -2*mp.re(mp.conj(p1)*p2),
        ]),
    }


def F_chi(xi_value):
    """mathcal F_chi(xi) in equation (18)."""
    # Fresh dictionaries are intentional: mpmath.hyper2d mutates them.
    return mp.hyper2d(
        {"m+n": [mp.mpf("0.5")], "m": [mu], "n": [mp.mpf("0.5") - mu]},
        {"m+n": [mp.mpf("1.5")], "m": [mp.mpf("0.5")], "n": [mp.mpf("0.5")]},
        zeta(xi_value), -zeta(xi_value),
    )


def G_chi(xi_value):
    """mathcal G_chi(xi) in equation (19)."""
    # Fresh dictionaries are intentional: mpmath.hyper2d mutates them.
    return mp.hyper2d(
        {"m+n": [mp.mpf(1)], "m": [mu], "n": [1 - mu]},
        {"m+n": [mp.mpf(2)], "m": [mp.mpf("1.5")], "n": [mp.mpf("0.5")]},
        zeta(xi_value), -zeta(xi_value),
    )


def J_chi(xi_value):
    """J_chi in equation (21), with J_chi'=-d as in equation (16)."""
    Ux = U(xi_value)
    Vx = V(xi_value)
    J_perp = (
        2*xi_value*F_chi(xi_value)
        - xi_value*mp.exp(-zeta(xi_value))*Ux*Vx
    )
    J3 = mp.re(ii*chi*xi_value**2*G_chi(xi_value)/2)
    return mp.matrix([mp.re(J_perp), mp.im(J_perp), J3])


def cartesian_bracket(xi_value):
    """The bracketed expressions in equations (22)--(23)."""
    Ux = U(xi_value)
    Vx = V(xi_value)

    x_plus_iy = (
        2*c*xi_value*F_chi(xi_value)
        + mp.exp(-zeta(xi_value)) * (
            a*chi*Ux**2
            + a*chi**3*xi_value**2*Vx**2/4
            - (c*xi_value + ii*a*chi*xi_value**2)*Ux*Vx
        )
    )
    z = (
        mp.re(ii*c*chi*xi_value**2*G_chi(xi_value)/2)
        + a*xi_value*(
            2*abs(Ux)**2 - 1 - chi**2*mp.im(mp.conj(Ux)*Vx)
        )
    )
    return mp.matrix([mp.re(x_plus_iy), mp.im(x_plus_iy), z])


# Endpoint subtraction in Theorem 1 gives r(0)=0.
cartesian_bracket_xi0 = cartesian_bracket(xi0)


def r(s):
    """The Cartesian parametrization r(s) in (22)--(23)."""
    return rho**(-mp.mpf("1.5")) * (
        cartesian_bracket(xi(s)) - cartesian_bracket_xi0
    )


# Values at xi_0 entering the basepoint formula (42).
frame_xi0 = sabban_frame(xi0)
J_chi_xi0 = J_chi(xi0)


def u_C_b(sign):
    """u_+/- , C_+/- , and b_+/- from (30), (40)--(42)."""
    u_perp = sign * (
        chi*mp.pi*mp.exp(-mp.pi*chi**2/8)*mp.exp(-mp.pi*ii/4)
        / (
            mp.sqrt(2)
            * mp.gamma(mp.mpf("0.5") + ii*chi**2/8)
            * mp.gamma(1 - ii*chi**2/8)
        )
    )
    u_pm = mp.matrix([
        mp.re(u_perp),
        mp.im(u_perp),
        mp.exp(-mp.pi*chi**2/4),
    ])

    C_perp = u_perp * (
        ii/chi
        - chi/4 * (
            2*mp.log(2)
            + mp.digamma(1 - ii*chi**2/8)
            + mp.digamma(mp.mpf("0.5") + ii*chi**2/8)
        )
    )
    C3 = chi/4 * (
        (1 - mp.exp(-mp.pi*chi**2/4))
        * (mp.log(2) + mp.re(mp.digamma(1 - ii*chi**2/8)))
        - (1 + mp.exp(-mp.pi*chi**2/4))
        * (mp.log(2) + mp.re(mp.digamma(mp.mpf("0.5") + ii*chi**2/8)))
    )
    C_pm = mp.matrix([mp.re(C_perp), mp.im(C_perp), C3])

    b_pm = rho**(-mp.mpf("1.5")) * (
        c*C_pm
        - c*J_chi_xi0
        - a*(xi0*frame_xi0["gamma"] - chi*frame_xi0["d"])
    )
    return u_pm, C_pm, b_pm


def xyz(v):
    """Convert an mpmath vector to a NumPy vector for plotting."""
    return np.array([float(v[0]), float(v[1]), float(v[2])])


def r_float(s):
    """One direct arbitrary-precision evaluation of r(s), converted to float."""
    return xyz(r(mp.mpf(str(s))))


def cuboid_edges(lower, upper):
    """The twelve equal-weight edges of the plotting cuboid."""
    x0, y0, z0 = lower
    x1, y1, z1 = upper
    vertices = np.array([
        [x0, y0, z0], [x1, y0, z0], [x1, y1, z0], [x0, y1, z0],
        [x0, y0, z1], [x1, y0, z1], [x1, y1, z1], [x0, y1, z1],
    ])
    pairs = [
        (0, 1), (1, 2), (2, 3), (3, 0),
        (4, 5), (5, 6), (6, 7), (7, 4),
        (0, 4), (1, 5), (2, 6), (3, 7),
    ]
    return [(vertices[i], vertices[j]) for i, j in pairs]


def main():
    # Each plotted point is a direct evaluation of equations (22)--(23).
    with multiprocessing.Pool(processes=min(5, multiprocessing.cpu_count())) as pool:
        curve = np.vstack(pool.map(r_float, s_values, chunksize=4))

    u_plus, _C_plus, b_plus = map(xyz, u_C_b(+1))
    u_minus, _C_minus, b_minus = map(xyz, u_C_b(-1))

    line_parameter = np.linspace(-1.6, 1.6, 240)
    line_plus = b_plus + line_parameter[:, None]*u_plus
    line_minus = b_minus + line_parameter[:, None]*u_minus
    asymptotes = np.vstack([line_minus, np.full((1, 3), np.nan), line_plus])

    geometry = np.vstack([curve, line_plus, line_minus])
    lower = geometry.min(axis=0)
    upper = geometry.max(axis=0)
    span = upper - lower
    lower -= np.array([0.10, 0.10, 0.07])*span
    upper += np.array([0.10, 0.10, 0.07])*span

    fig = plt.figure(figsize=(8, 6))
    ax = fig.add_subplot(111, projection="3d")
    ax.plot(curve[:, 0], curve[:, 1], curve[:, 2], linewidth=1.8)
    ax.plot(
        asymptotes[:, 0], asymptotes[:, 1], asymptotes[:, 2],
        "--", linewidth=1.1,
    )

    ax.set_axis_off()
    for p, q in cuboid_edges(lower, upper):
        ax.plot(
            [p[0], q[0]], [p[1], q[1]], [p[2], q[2]],
            color=plt.rcParams["axes.edgecolor"], linewidth=0.65,
        )

    ax.set_xlim(lower[0], upper[0])
    ax.set_ylim(lower[1], upper[1])
    ax.set_zlim(lower[2], upper[2])
    ax.set_box_aspect((upper - lower)*np.array([1.12, 1.0, 0.90]), zoom=0.91)
    ax.view_init(elev=17, azim=-58)
    ax.set_proj_type("persp", focal_length=0.95)
    fig.tight_layout(pad=0)

    out = Path(__file__).with_name("figure1")
    fig.savefig(out.with_suffix(".pdf"), bbox_inches="tight", pad_inches=0.02)
    fig.savefig(out.with_suffix(".png"), dpi=300, bbox_inches="tight", pad_inches=0.02)
    np.savetxt(
        out.with_suffix(".csv"),
        np.column_stack([s_values, curve]),
        delimiter=",",
        header="s,x,y,z",
        comments="",
    )


if __name__ == "__main__":
    main()
\end{lstlisting}

\end{document}